\documentclass[11pt,a4paper,leqno]{amsart}
\usepackage[utf8]{inputenc}
\nonstopmode \textwidth=16.00cm
\usepackage{amssymb,color,tikz, cancel}
\usepackage{amsmath,amscd,amsfonts,amsthm,verbatim}
\usepackage{tikz-cd}
\usepackage{enumitem}
\usepackage[all]{xy}
\usetikzlibrary{matrix} 

\usepackage{xcolor}
\usepackage{blindtext}
\usepackage{multicol}

\numberwithin{equation}{section} 

\definecolor{ffzzqq}{rgb}{1,0.6,0}
\definecolor{qqqqff}{rgb}{0,0,1}
\definecolor{ffqqqq}{rgb}{1,0,0}
\definecolor{wwzzqq}{rgb}{0.4,0.6,0}
\definecolor{zzwwff}{rgb}{0.6,0.4,1}
\usepackage{pgfplots}
\pgfplotsset{compat=1.15}
\usepackage{mathrsfs}
\usetikzlibrary{arrows}

\newcommand {\PP}{\mathbb{P}}

\newcommand {\sT}{\mathcal{T}}
\newcommand {\sO}{\mathcal{O}}

\newcommand {\sE}{\mathcal{E}}
\newcommand {\sF}{\mathcal{F}}

 \def\cocoa{{\hbox{\rm C\kern-.13em
   o\kern-.07em C\kern-.13em o\kern-.15em A}}}
\newtheorem{theorem}{Theorem}[section]
\newtheorem{lemma}[theorem]{Lemma}
\newtheorem{proposition}[theorem]{Proposition}
\newtheorem{corollary}[theorem]{Corollary}
\newtheorem{question}[theorem]{Question}
\newtheorem{problem}[theorem]{Problem}
 \theoremstyle{definition}
\newtheorem{definition}[theorem]{Definition} \theoremstyle{remark}
\newtheorem{remark}[theorem]{Remark}

\definecolor{MyDarkGreen}{cmyk}{0.7,0,1,0}

\begin{document}

\title[WLP of equigenerated complete intersections. Applications]{Weak Lefschetz property of equigenerated complete intersections. Applications}
 \author[V. Beorchia]{Valentina Beorchia} 
 \address{Dipartimento di Matematica e Geoscienze, Universit\`a di
Trieste, Via Valerio 12/1, 34127 Trieste, Italy}
 \email{beorchia@units.it, 
 ORCID 0000-0003-3681-9045}.
 \author[R.\ M.\ Mir\'o-Roig]{Rosa M.\ Mir\'o-Roig} 
 \address{Facultat de
 Matem\`atiques i Inform\`atica, Universitat de Barcelona, Gran Via des les
 Corts Catalanes 585, 08007 Barcelona, Spain} \email{miro@ub.edu, ORCID 0000-0003-1375-6547}

\thanks{The first author is a member of GNSAGA of INdAM, is supported by MUR funds: PRIN project GEOMETRY OF ALGEBRAIC STRUCTURES: MODULI, INVARIANTS, DEFORMATIONS, PI Ugo Bruzzo, Project code 2022BTA242, and by the University of Trieste project FRA 2025.} 
\thanks{The second author has been partially supported by the grant PID2020-113674GB-I00}

\thanks{{\bf 2020 Mathematics Subject Classification} Primary 14M10; Secondary 14F06, 13E10}

\medskip \noindent
\thanks{{\bf Keywords}. Complete intersection, Lefschetz properties, Artinian algebras, semistability}

\begin{abstract} In this paper, we prove that any Artinian complete intersection homogeneous ideal $I$ in $K[x_0,\cdots,x_n]$ generated by 
$n+1$ forms of degree $d\ge 2$ satisfies the weak Lefschetz property (WLP) in degree $t< d+\lceil \frac{d}{n} \rceil$. As a consequence, we get that the Jacobian ideal of a smooth 3-fold of degree $d\ge 7$ in $\PP^4$ satisfies the weak Lefschetz property in degree $d$, answering a recent question of Beauville \cite{B}. 
\end{abstract}

\maketitle

\section{Introduction}

The present paper concerns the weak Lefschetz property for graded Artinian $K$-algebras $A$, that is when is
the multiplication map $\times \ell:[A]_{i-1}\longleftrightarrow [A]_i$
with a general linear form $\ell \in [A]_1$ has maximal rank for any integer $i>0$.

This property plays a crucial role in the study of
Artinian $K$-algebras, influencing their algebraic structure and leading to unexpected geometric applications. To determine whether an Artinian standard graded $K$-algebra $A$ has the WLP seems a simple problem of linear algebra, but it has proven to be extremely elusive and much more work on this topic remains to be done, see \cite{JMR} and its reference list for more information. 
In particular, the following longstanding problem remains open:

\begin{problem}
  Does {\em any} complete intersection Artinian graded $K$-algebra satisfy WLP?
\end{problem}

When $I = (f_0,\cdots ,f_n)$ is a complete intersection, we know that for a general choice of
the homogeneous forms $f_i$, $0\le i \le n$, $R/I$ has the WLP thanks to a result of \cite{S} and \cite{W}; and it has been asked and conjectured
by many authors whether {\em every} complete intersection has the WLP (and also whether
every complete intersection has the related Strong Lefschetz Property, SLP).
This question was solved in codimension 2 and 3 in \cite{HMNW} and remains open in higher codimension. In codimension 4 and for equigenerated complete intersections, the injectivity does hold in
a very satisfying range \cite{BMMN} and for codimension greater than 5 and equigenerated complete intersections injectivity is only known in the first spot (namely, from degree $d-1$ to degree $d$) where WLP can fail. For more details on these partial results and the state of the art on this question the reader can look at \cite{AR}, \cite{BMMN} and \cite{JMR}.

In five or more variables very little is known beyond the results mentioned above. In
this short note, we consider the case of an arbitrary number of variables and we first deal with the case where all generators have the same degree $d$, although our last result is for
arbitrary degrees $d_0,\cdots, d_n$. We begin by translating the problem to a more general issue: to understand the behaviour of a (stable) vector bundle $\sE$ on $\PP^n$ when we restrict it to a general linear subspace of $\PP^n$. As a main tool we use Grauert-M\"ulich theorem and Flenner's restriction theorem. The question of the stability of the syzygy bundle associated with an Artinian algebra has been investigated also in the case of monomial ideals in \cite{CN}.

In the last section we apply our results to the case of the Jacobian ideal of a smooth hypersurface and to study the variation of the hyperplane section of a smooth hypersurface $X\subset \PP^n$. In particular, we answer a recent question posed by Beauville in \cite{B}.

While the characteristic of the ground field plays a very interesting role in such questions, in
this paper we work over a field of characteristic zero because as one of the main tools we use Grauert-M\"ulich theorem which only works in characteristic 0.

\section{Preliminaries}

In this section, we fix the notation, we recall the basic definitions and we gather together the main results needed later on. 
Throughout this paper $K$ will be an algebraically closed field of characteristic zero, $R=K[x_0,\dots,x_n]$ and $\PP^n={\rm Proj}(R)$.

\begin{definition} \label{WLP def}
Let $A=R/I$ be a graded Artinian $K$-algebra. We say that $A$ has the {\em weak Lefschetz property} (WLP, for short)
if there is a linear form $\ell \in [A]_1$ such that, for all
integers $i> 0$, the multiplication map
\[
\times \ell: [A]_{i-1} \longrightarrow [A]_{i}
\]
has maximal rank.
 In this case, the linear form $\ell$ is called a {\it (weak) Lefschetz element} of $A$. If for the general linear form $\ell \in [A]_1$ and for an integer $j$ the map $\times \ell:[A]_{j-1} \longrightarrow [A]_{j}$ does not have maximal rank, we will say that the ideal $I$ fails the WLP in degree $j$.
\end{definition}

In this short note we focus our attention on Artinian complete intersection graded $K$-algebras.
 Due to their duality properties, the WLP for Artinian complete intersection $K$-algebras is determined by a single map. Indeed, we have:

\begin{lemma}\label{mid map}
 Let $A$ be an Artinian Gorenstein $K$-algebra with $reg(A)=e$. Then the following are equivalent 
  \begin{enumerate}[label=(\roman*),itemsep=0.1cm,topsep=0.1cm]
    \item $\ell\in A_1$ is a weak Lefschetz element on $A$ 
    \item the map $\times \ell: A_{\lfloor\frac{e-1}{2}\rfloor} \to A_{\lfloor\frac{e-1}{2}\rfloor+1}$ has maximum rank (is injective) 
    \item the map $\times \ell: A_{\lfloor\frac{e}{2}\rfloor} \to A_{\lfloor\frac{e}{2}\rfloor+1}$ has maximum rank (is surjective).
    \end{enumerate}
\end{lemma}
\begin{proof}
  See \cite[Proposition 2.1]{MMN}.
\end{proof}

 In the proof of our main results vector bundle theory will play an important role. For sake of completeness we recall here the main results that we will use.

Given a torsion-free sheaf $\sE$ on $\PP^n$, we define its {\em slope} $\mu(\sE)$ as the quotient $c_1(\sE)/rk(\sE)$ and we say that 
$\sE$ is {\em semistable} (respectively, {\em stable}) if there
is no proper subsheaf $\sF$ of $\sE$ with $\mu(\sF) > \mu(\sE)$ (respectively, $\mu(\sF) \ge \mu(\sE)$).

The following criterion due to Bohnhorst and Spindler’s will allow us to check the 
 stability of a syzygy bundle associated to an equigenerated complete intersection Artinian ideal.

\begin{proposition}\label{stable}
 Let $\sE $ be a vector bundle of rank $n$ on $\PP^n$ fitting into a short exact sequence
\[\xymatrix{0\ar[r]&
  {\bigoplus\limits_{i=1}^k\sO(a_i)}\ar[r]&
  {\bigoplus\limits_{j=1}^{n+k}\sO(b_j)}
  \ar[r]& \sE\ar[r]&0\mbox{,}}\]
with ${a_1\ge\cdots\ge a_k}$ and ${b_1\ge\cdots\ge b_{n+k}}$. 

Then $\sE$ is stable if and only if
\[b_1<\mu(E):=\frac{c_1(\sE)}{rk(\sE)}=\tfrac{1}{n}\left[\sum_{j=1}^{n+k}b_j
  -\sum_{i=1}^ka_i\right].\]
\end{proposition}
\begin{proof}
  See \cite[Theorem 2.7]{BS}.
\end{proof}

The other key ingredient is the so-called Grauert-M\"ulich theorem, which controls the splitting type of a semistable rank $r$ vector bundle $\sE$ on $\PP^n$ when we restrict to a general line $L\subset \PP^n$.

\begin{theorem}\label{splitting} 
Let $\sE$ be a rank $r$ semistable vector bundle on $\PP^n$ and let $L\subset \PP^n$ be a general line. If $\sE _{|L}\cong \oplus _{i=1}^r \sO_L(b_i)$ with $b_1\ge b_2\ge \cdots \ge b_r$ then
$$
0\le b_i-b_{i+1}\le 1
$$
for all $i=1,\cdots, r-1.$
\end{theorem}
\begin{proof}
  See \cite[Theorem 2.1.4]{OSS}.
\end{proof}


\section{Weak Lefschetz property of equigenerated complete intersections}

This section is entirely devoted to the study of the WLP for equigenerated complete intersections Artinian $K$-algebras. Since WLP is well known for any Artinian $K$-algebras of codimension 2 \cite{Br} and for any Artinian complete intersection $K$-algebras of codimension 3 \cite{HMNW}, we will assume that the codimension is always $\ge 4$. 

Therefore, from now on, we fix integers $n\ge 3$ and $d\ge 2$ and we write $d=an+b$ with $0\le b<n$.
Let $I = (f_0, f_1,\cdots , f_n)\subset R$
be a homogeneous complete intersection ideal, with $deg f_i = d$ for $0\le i\le n$.
We consider the
  syzygy bundle
  \begin{equation}\label{eq: syzygy bundle E}
  {\sE}:= \ker \left ( {\mathcal O}^{\oplus n+1}_{\PP^n}
  \stackrel{(f_0,\cdots, f_n)}{\longrightarrow} {\mathcal
   O}_{\PP^n}(d) \right )
  \end{equation}
  associated to $(f_0,\cdots ,f_n)$. The vector bundle ${\sE}$ has rank $n$
   on $\PP^n$ and $c_1( {\sE})=-d$. Note that 
   $$
   H^1_*((\sE(-d)) = \bigoplus_{t \in \mathbb Z} H^1(\sE(t-d)) \cong R/I
   $$ 
   (cf.\ \cite[Proposition 2.1]{BK}). Let us check that ${\sE}$ is stable. To this end, we consider the exact sequence
  \[
  0 \longrightarrow {\mathcal O}_{\PP^n}(-d) \longrightarrow {\mathcal O}_{\PP^3}^{n+1} \longrightarrow {\mathcal
   E}^\ast\longrightarrow 0.
  \]
  By construction we have
  \[
  0<\mu({\mathcal E}^\ast):=\frac{c_1({\mathcal E}^\ast)}{rk({\mathcal
    E}^\ast)}=\frac{d}{n}.
  \]
  So we can apply Proposition \ref{stable}, and conclude that
  ${\mathcal E}^\ast$ is stable. Since stability is
  preserved under dualizing, we also have that ${\mathcal E}$ is stable.

  Since ${\mathcal E}$ is stable, we can apply Theorem \ref{splitting} and compute the splitting type of $\sE$ on a general line $L\subset \PP^n$. In fact, for a general line $L\subset \PP^n$ we have $\sE _{|L}\cong \oplus _{i=1}^n \sO_L(b_i)$ with $b_1\ge b_2\ge \cdots \ge b_n$,
$0\le b_i-b_{i+1}\le 1$ for $i=1, \cdots , n-1$ and $\sum_{i=1}^rb_i=-d.$ Therefore, we have
\begin{equation}\label{keybounds}
-a-\lfloor\frac{n+1}{2}\rfloor\le b_1 \quad \text{ and } \quad b_n\le -a+ \lfloor\frac{n-1}{2}\rfloor .
\end{equation}

\begin{theorem}\label{main} 
Fix integers $n\ge 3$ and $d\ge 2$. Let $I=(f_0,f_1,\cdots,f_n)\subset R$ a complete intersection ideal with $deg(f_i)=d$, $0\le i \le n$. 

If 
$d-1< t< d+ \lfloor\frac{d}{n} \rfloor- \lfloor\frac{n-1}{2}\rfloor$,
then $I$ satisfies the WLP in degree $t$.
  \end{theorem}
  
  \begin{proof} We have to check that there is a linear form $\ell \in R/I$ such that 
  $$
  \times \ell : [R/I]_{t-1}\longrightarrow [R/I]_t
  $$ 
  has maximal rank. 
  
    Let $\sE $ be the rank $n$, stable vector bundle with $c_1(\sE)=-d$ associated to the regular sequence $(f_0,f_1,\cdots,f_n)$. As above we write $d=an+b$ with $0\le b<n$. By (\ref{keybounds}) we know that the splitting type of $\sE$ on a general line $L\subset \PP^n$ is $\sE _{|L}\cong \oplus _{i=1}^n \sO_L(b_i)$ with $-a-\lfloor\frac{n+1}{2}\rfloor\le b_1 \le b_2\cdots \le b_n\le -a+ \lfloor\frac{n-1}{2}\rfloor $. This implies that for any general line $L\subset \PP^n$ it holds:

\begin{equation}\label{vanishing}
   H^0(L,\sE_{|L}(t))=0 \text{ for all } t<a- \lfloor\frac{n-1}{2}\rfloor.
    \end{equation}
 We consider the flag $L\cong\PP^1\subset \PP^2\subset \cdots\subset \PP^{n-1}\subset \PP^n$ and for any $i$ with $2\le i\le n$ the short exact sequence:
 $$
 0\longrightarrow \sE_{|\PP^i}(-1) \longrightarrow \sE_{|\PP^i}\longrightarrow \sE_{|\PP^{i-1}}\longrightarrow 0.
 $$
 Taking cohomology and using the equality (\ref{vanishing}), we conclude that for a general hyperplane $H\subset \PP^n$ we also have 
 \begin{equation}\label{vanishing2}
   H^0(H,\sE_{|H}(t))=0 \text{ for all } t<a- \lfloor\frac{n-1}{2}\rfloor .
    \end{equation}
    Since $A=R/I \cong H^1_{*}(\sE(-d))$, we have that $R/I$ satisfies the WLP in degree t if and only if, for a general linear form $\ell \in [A]_1$, the multiplication map
    $$
    \times \ell: H^1(\PP^n,\sE(t-d-1))\longrightarrow H^1(\PP^n,\sE (t-d))
    $$ 
    has maximal rank. This is automatically satisfied if for a general hyperplane $H\subset \PP^n$ either $H^0(H,\sE_{|H}(t-d))=0$ or $H^1(H,\sE_{|H}(t-d))=0$. Therefore, we conclude that $I$ satisfies the WLP in degree $t$ for all 
    $t\le d+a- \lfloor\frac{n-1}{2}\rfloor $, which proves what we want.
  \end{proof}

\begin{remark}
  The case $t=d$ follows from \cite[Corollary 4]{AR}. 
    
\end{remark}
\begin{remark}
 For $n>3$, the result is completely new while for $n=3$ our result is weaker that the result in \cite[Theorem 4.9]{BMMN} where the authors prove that for $n=3$ and $I=(f_0,f_1,f_2,f_3)\subset K[x,y,z,t]$ a complete intersection ideal generated by 4 forms of degree  $d$, the WLP holds for any $t< \lfloor \frac{3d+1}{2} \rfloor $.
\end{remark}

Using Flenner's result about the fact that (semi)stability is preserved when we restrict to a general hyperplane, we can improve a little bit the previous bound and get a new range where injectivity holds.

 \begin{theorem} \label{bound2}
 Fix integers $n\ge 3$ and $d\ge 2$. Let $I=(f_0,f_1,\cdots,f_n)\subset R$ be a complete intersection ideal with $deg(f_i)=d$, $0\le i \le n$.
    If $t \le d+ \lceil\frac{d}{n}\rceil$, then $I$ satisfies the WLP in degree $t$. 
 \end{theorem}
 
\begin{proof} We denote by $\sE$ the rank $n$ syzygy bundle associated to $I$. By \cite[Proposition 1.15]{Flenner},
for a general hyperplane $H\subset \PP^n$ the restriction
$\sE _{ |H}$ is not semistable if and only if $\sE$ is either isomorphic to a twist of $\Omega^1_{\PP^n}$ or of $\sT_{\PP^n}$. This is not the case since $\sE$ is defined by the sequence \eqref{eq: syzygy bundle E} and $d\ge 2$. 

By semistability we have
$$ 
H^0(H,\sE _{ |H}(t))=0 \text{ for all } t<\left\lceil \frac{d}{n}\right\rceil =\left\lceil \frac{-c_1(\sE _{ |H})}{rk(\sE _{ |H})}\right\rceil.
$$
 We now argue as in the proof of Theorem \ref{main} using this last vanishing result instead of (\ref{vanishing2}) and we get what we want.
 \end{proof}

 We conclude this section with the following result, which holds for any homogeneous Artinian ideal $I$. The idea is simple and relies on the fact that, if we know the generic splitting type of the syzygy bundle $\sE$ naturally associated with a presentation of $I$, we can always bound the first twist of $\sE$ which provides a section.

 \begin{proposition} Let $I=(f_0,f_1,\cdots,f_r)\subset R$ be an Artinian homogeneous ideal, with $d_i :={\rm deg} f_i$, and 
consider the
  syzygy bundle $\sE$ of rank $r$
  \begin{equation}\label{eq: syzygy sheaf E}
  {\sE}:= \ker \left ( \bigoplus_{i=0}^r{\mathcal O}_{\PP^n}(-d_i)
  \stackrel{(f_0,\cdots, f_r)}{\longrightarrow} {\mathcal
   O}_{\PP^n} \right )
  \end{equation}
  associated with $(f_0,\cdots ,f_r)$.
Let
$b_1 \ge b_2\cdots \ge b_r$ denote
the splitting type of $\sE$ on a general line $L\subset \PP^n$.

Then, if $t< -b_1$, the ideal $I$ satisfies the WLP in degree $t$.
 \end{proposition}

 \begin{proof}
   By arguing as in the proof of Theorem \ref{main}, we see that for
   a general hyperplane $H\subset \PP^n$ we have 
   $H^0(H,\sE_{|H}(t))=0$ for any $t<-b_1$.
    Since $A=R/I \cong H^1_{*}(\sE)$, we have that $R/I$ satisfies the WLP in degree $t$ if and only if, for a general linear form $\ell \in [A]_1$ the multiplication map
    $
    \times \ell: H^1(\PP^n,\sE(t-1))\longrightarrow H^1(\PP^n,\sE (t))
    $ has maximal rank, which automatically holds if for a general hyperplane $H\subset \PP^n$ we have $H^0(H,\sE_{|H}(t))=0$.
 \end{proof}

\section{Applications}

Our first goal is to apply our main result to the Jacobian ideal of a smooth hypersurface $X\subset \PP^n$ of degree $d$. If $X$ is defined by a homogeneous polynomial $f\in R$ then the Jacobian ideal $J(X)$ is the homogeneous ideal generated by the partial derivatives $\frac{\partial f}{\partial x_i}$, $0 \leq i \leq n$, and it is an Artinian complete intersection generated by forms of degree $d-1$ since $X$ is smooth.
In \cite{B}, Beauville posed the following question:

\begin{question}
 Does the Jacobian ideal of a smooth hypersurface of degree $d$ have the weak Lefschetz property in degree $d$?
\end{question}
\noindent In addition, he pointed out that an affirmative answer to the above question will have important geometric consequences that we will detail later in this section. More generally, in \cite{I} Ilardi asked the following question:

\begin{question}
 Does the Jacobian ideal of a smooth hypersurface have the WLP?
\end{question}

So far, very little is know, and we quickly sumarize what is known until now.

\begin{proposition}
Let $X : f = 0$ be a smooth hypersurface in $\mathbb P^n$ of degree $d > 2$. Then the Jacobian ideal $J(X)$ has the WLP in degree $d -1$.
\end{proposition}
\begin{proof}
 See \cite[Proposition 3.1]{I}.
\end{proof}

In \cite{AR}, Alzati and Re proved the WLP in degree $d$ for any complete intersection generated by forms of degree $d$, not only for Jacobian ideals while in \cite{BMMN}, Ilardi's result was generalized for smooth surfaces in $\PP^3$. Indeed, we have

\begin{proposition}\label{jacobian result}
 (1) Let $X : f = 0$ be a smooth surface in $\mathbb P^3$ of degree $3, 4, 5$ or $6$. Then the Jacobian ideal $J(X)$ has the WLP.

(2) 
Let $X : f = 0$ be a smooth surface in $\mathbb P^3$ of degree $d > 2$. Then the ideal $J(X)$ has the WLP in all degrees $\leq
 \lfloor\frac{3d+1}{2}\rfloor -2$.
\end{proposition}
\begin{proof}
 (1) See \cite[Corollary 7.2]{BMMN}.

 (2) See \cite[Corollary 7.3]{BMMN}.
\end{proof}

We remark that WLP for the Jacobian ideal of a smooth surface $X\subset \PP^3$ is equivalent to injectivity for all $t \leq 2d-3$, so Proposition \ref{jacobian result} covers approximately half the range left open by the Ilardi and Alzati-Re result.
As an immediate consequence of our main result in section 3 we get:

\begin{theorem} \label{jacobian result2}
 Let $X : f = 0$ be a smooth hypersurface in $\mathbb P^n$ with $n \ge 3$ of degree $d > 2$. Then the Jacobian ideal $J(X)$ has the WLP in any degree $t< d-1 + \lceil \frac{d-1}{n}\rceil $.
\end{theorem}
\begin{proof}
 It immediately follows from Theorem \ref{bound2} since the Jacobian ideal $J(X)$ of a smooth hypersurface $X\subset \PP^n$ of degree $d$ is an Artinian complete intersection ideal generated by forms of degree $d-1$.
\end{proof}

Following ideas recently developed by Beauville in \cite{B}, we will now apply our result on Jacobian ideals to study the variation of the hyperplane section of a smooth hypersurface $X\subset \PP^{n+1}$. To this end, we demote by ${\mathcal M}_d$ the moduli space of (smooth) hypersurfaces of degree $d$ in $\PP^n$. For any smooth hypersurface $X\subset \PP^{n+1}$ of equation $F=0$, we denote by $X^*\subset (\PP^{n+1})^*$ its dual and we consider the morphism
$$
s_X:(\PP^{n+1})^*\setminus X^*\longrightarrow {\mathcal M}_d\cong \PP(H^0(\sO_{\PP^{n}}(d))/SL_{n+1}, \quad H\mapsto X\cap H.
$$

In \cite{HMP}, Harris, Mazur and Pandharipande posed the following question: 
\begin{question}
 Let $ X\subset \PP^{n+1}$ be a smooth, hypersurface of degree $d$. Does the family of all smooth hyperplane sections of X vary maximally in moduli? or, equivalently, is the morphism $s_X$ generically finite onto its image?
\end{question}

\begin{corollary} \label{cor: maximal variation}
Let $ X\subset \PP^{n+1}$ with $n \ge 2$ be a smooth, hypersurface of degree $d$. For all $d\ge n+3$ the morphism $s_X$ is generically finite onto its image.
\end{corollary}
\begin{proof}
  By Theorem \ref{jacobian result2} the Jacobian ideal $J(X)$ of the hypersurface $X$ satisfies the WLP in degree 
  $t< d-1 + \lceil \frac{d-1}{n+1}\rceil $. In particular, since $d\ge n+3$, the ideal $J(X)$ verifies the WLP in degree $d$, i.e. the morphism 
  $$
  \times L: [K[x_0,\cdots ,x_{n+1}]/J(X)]_{d-1}\longrightarrow [K[x_0,\cdots ,x_{n+1}]/J(X)]_{d}
  $$ 
  is injective. As observed by Beauville in \cite{B}, such a property is
  equivalent to $s_X$ being unramified at $H=V(L)$, since the proof of
  \cite[Proposition]{B} applies to any degree $d \ge 3$. It follows that $s_X$ is generically finite onto its image.
\end{proof}

\begin{remark}
  Corollary \ref{cor: maximal variation} has also been proved in \cite[Theorem 1.3]{PRT} using a different approach. 
\end{remark}


\begin{thebibliography}{00}

\bibitem{AR} A. Alzati and R. Re, {\it Complete intersections of quadrics and the weak Lefschetz property}, Collect. Math 70 (2019), no. 2, 283--294.

\bibitem{B} A. Beauville, {\em Hyperplane section of cubics threefolds}, preprint arXiv 2501.07586.

\bibitem{BS} G. Bohnhorst and H. Spindler, {\em The stability of certain vector
bundles on $\PP^n$}, Lecture Notes in Math. {\bf 1507} (1992), 39–50.

\bibitem{BK} H. Brenner and A. Kaid, {\it Syzygy bundles on $\mathbb P^2$ and the Weak Lefschetz Property}, Illinois J. Math. 51 (04) (2007), 1299--1308.

\bibitem{BMMN} M. Boij, J. Migliore, R.M. Mir\'o-Roig and U. Nagel. {\em The weak Lefschetz property for height four complete intersections}. Trans AMS, {\bf 10} (2023), 1254-1286.


\bibitem{Br} J. Brian\c con, Description de HilbnC[x, y], Invent. Math. {\bf 41} (1977), 45–89.

\bibitem{CN} D. Cook II, U. Nagel, {\em Syzygy bundles and the weak Lefschetz property of monomial almost complete intersections},
J. Commut. Algebra 13(2) (2021), 157-178.

\bibitem{Flenner} H. Flenner, {\em Restrictions of semistable bundles on projective varieties}, Commentarii mathematici Helvetici 59 (1984), 635-650.

\bibitem{HMNW} T. Harima, J. Migliore, U. Nagel and J. Watanabe, {\em The Weak and Strong Lefschetz properties for artinian
K-algebras}, J. Algebra {\bf 262} (2003), 99–126.

\bibitem{HMP} J. Harris, B. Mazur, and R. Pandharipande,{\em Hypersurfaces of low degree}, Duke Math. J.,
{\bf 95} (1998) 125-160.

\bibitem{I} G. Ilardi. \newblock {\it Jacobian ideals, arrangements and the Lefschetz properties.} \newblock { J. Algebra}, 508:418--430, 2018.


\bibitem{JMR}M. Juhnke-Kubitzke and R. M. Miró-Roig. {\em List of problems. } In: Nagel, U., Adiprasito,
K., Di Gennaro, R., Faridi, S., Murai, S. (eds) Lefschetz Properties. SLP-WLP 2022. Springer INdAM
Series, vol 59. Springer, Singapore. https://doi.org/10.1007/978-981-97-3886-1-12.

\bibitem{MMN} J. C. Migliore, M. Miró-Roig, and U. Nagel. {\em Monomial ideals, almost complete intersections and the
weak lefschetz property}, Trans AMS, {\bf 363} (2011) 229 – 257.

\bibitem{OSS} C. Okonek, M. Schneider, H. Spindler, {\em Vector bundles on complex projective spaces}, Progress in Math. {\bf 3} (1980), Birkh\"{a}user.

\bibitem{PRT} A.~P. Patel, E. Riedl and D. Tseng, {\em Moduli of linear slices of high degree smooth hypersurfaces}, Algebra Number Theory {\bf 18} (2024), no.~12, 2133--2156. 

\bibitem{S} R. P. Stanley, {\em Weyl groups, the hard Lefschetz theorem, and the Sperner property}. SIAM
J. Algebraic Discrete Methods, {\bf 1} (1980), 168–184.

\bibitem{W} J. Watanabe, {\em The Dilworth number of Artinian rings and finite posets with rank function.}
In Commutative algebra and combinatorics (Kyoto, 1985), volume 11 of Adv. Stud. Pure Math., pages
303–312. North-Holland, Amsterdam, 1987.

\end{thebibliography}
\end{document}